\documentclass[reqno,b5paper]{amsart}
\usepackage{amsmath}
\usepackage{amssymb}
\usepackage{amsthm}
\usepackage{enumerate}
\usepackage[mathscr]{eucal}
\usepackage{graphicx}
\newtheorem{theorem}{Theorem}[section]
\newtheorem{lemma}[theorem]{Lemma}
\newtheorem{proposition}[theorem]{Proposition}
\newtheorem{corollary}[theorem]{Corollary}

\theoremstyle{definition}

\begin{document}
\title[The open-point and bi-point-open topologies on $C(X)$]{Completeness Properties of the open-point and bi-point-open topologies on $C(X)$}
\author{Anubha Jindal}
\address{Anubha Jindal: Department Of Mathematics, Indian Institute of Technology Delhi,
New Delhi 110016, India.}
\email{jindalanubha217@gmail.com}

\author{R. A. McCoy}
\address{R. A. McCoy: Department of Mathematics, Virginia Tech, Blacksburg VA 24061-0123, U.S.A.}
\email{mccoy@math.vt.edu}

\author{S. Kundu}
\address{S. Kundu:  Department Of Mathematics, Indian Institute of Technology Delhi,
New Delhi 110016, India.}
\email{skundu@maths.iitd.ac.in}

\author{Varun Jindal}
\address{Varun Jindal: Department Of Mathematics, Malaviya National Institute of Technology Jaipur,
Jaipur, Rajasthan 302017, India.}
\email{vjindal.maths@mnit.ac.in}

\begin{abstract}
This paper studies various completeness properties of the open-point and bi-point-open topologies on the space $C(X)$ of all real-valued continuous functions on a Tychonoff space $X$. The properties range from complete metrizability to the Baire space property.
\end{abstract}
\maketitle
\section{\bf Introduction}
The set $C(X)$ of all real-valued continuous functions on a Tychonoff space $X$ has a number of natural topologies. One important topology on $C(X)$ is a set-open topology, introduced by Arens and Dugundji in \cite{AD}. Among the set-open topologies on $C(X)$, the point-open and compact-open topologies are most useful and frequently studied from different perspectives. The point-open topology $p$ is also known as the topology of pointwise convergence. The study of pointwise convergence of sequences of functions is as old as the calculus. The topological properties of $C_p(X)$, the space $C(X)$ equipped with the topology $p$, have been studied particularly in \cite{arhan}, \cite{mn}, \cite{tka}, \cite{tka1} and \cite{tka2}.

In the definition of a set-open topology on $C(X)$, we use a certain family of subsets of $X$ and open subsets of $\mathbb{R}$. Occasionally, there have been attempts, such as in \cite{pk}, to define a new type of set-open topology on $C(X)$. But even these attempts did not help to move much away from the traditional way of defining the set-open topologies on $C(X)$. So in \cite{AMK}, by adopting a radically different approach, two new kinds of topologies called the open-point and bi-point-open topologies on $C(X)$ have been defined. One main reason for adopting such a different approach is to ensure that both $X$ and $\mathbb{R}$ play equally significant roles in the construction of topologies on $C(X)$. This gives a function space where the functions get more involved in the behavior of the topology defined on $C(X)$.

The point-open topology on $C(X)$ has a subbase consisting of sets of the form
$$[x, V]^+ = \{f \in C(X) : f(x) \in V \},$$
where $x\in X$ and $V$ is an open subset of $\mathbb{R}$. The open-point topology on $C(X)$
has a subbase consisting of sets of the form $$[U, r]^- = \{f \in C(X) : f^{-1}(r) \cap U\neq\emptyset\},$$ where $U$ is an open subset of $X$ and $r \in \mathbb{R}$. The open-point topology on $C(X)$ is denoted by $h$ and
the space $C(X)$ equipped with the open-point topology $h$ is denoted by $C_h(X)$. The term $``h"$ comes from the word
``horizontal" because, in the case of $C_h(\mathbb{R})$ a subbasic open set can be viewed as the set of functions in $C(\mathbb{R})$ whose
graphs pass through some given horizontal open segment in $\mathbb{R}\times \mathbb{R}$, as opposed to a subbasic open set
in $C_p(\mathbb{R})$ which consists of the set of functions in $C(\mathbb{R})$ whose graphs pass through some given vertical
open segment in $\mathbb{R} \times \mathbb{R}$.

The bi-point-open topology on $C(X)$ is the join of the point-open topology $p$ and the open-point topology $h$. In other words, it is the topology having subbasic open sets of both kinds: $[x,V]^{+}$ and $[U,r]^{-}$, where $x\in X$ and $V$ is an open subset of $\mathbb{R}$, while $U$ is an open subset of $X$ and $r \in \mathbb{R}$. The bi-point-open topology on the space $C(X)$ is denoted by $ph$ and the space $C(X)$ equipped with the bi-point-open topology $ph$ is denoted by $C_{ph}(X)$. One can also view the bi-point-open topology on $C(X)$ as the weak topology on $C(X)$ generated by the identity maps $id_1:C(X)\rightarrow C_p(X)$ and $id_2:C(X)\rightarrow C_h(X)$.

In \cite{AMK} and \cite{AMK1}, in addition to studying some basic properties of the spaces $C_h(X)$ and $C_{ph}(X)$, the submetrizability, metrizability, separability and some cardinal functions on $C_h(X)$ and $C_{ph}(X)$ have been studied. But another important family of properties, the completeness properties of $C_{h}(X)$ and $C_{ph}(X)$, is yet to be studied. In this paper, we plan to do exactly that. More precisely, we study the completeness properties such as complete metrizability, $\check{C}$ech-completeness, local $\check{C}$ech-completeness, sieve-completeness, partition-completeness, pseudocompleteness and property of being a Baire space of the spaces $C_h(X)$ and $C_{ph}(X)$. We see that the completeness of $C_h(X)$ and $C_{ph}(X)$ behave like that of $C_p(X)$. Here we would like to recall that the completeness properties of $C_p(X)$ were studied in \cite{LM}.

In this paper, we use the following conventions. The symbols $\mathbb{R}$, $\mathbb{Q}$, $\mathbb{Z}$ and $\mathbb{N}$ denote the space of real numbers, rational numbers, integers and natural numbers, respectively. For a space $X$ the symbol $X^0$ denotes the set of all isolated points in $X$, $|X|$ denotes the cardinality of the space $X$, $\overline{A}$ denotes the closure of $A$ in $X$, $A^c$ denotes the complement of $A$ in $X$, and $0_X$ denotes the constant zero function in $C(X)$. Also for any two topological spaces $X$ and $Y$ that have the same underlying set, the three expressions $X=Y$, $X \leq Y$ and $X < Y$ mean that, respectively, the topology of $X$ is same as topology of $Y$, the topology of $X$ is weaker than or equal to topology of $Y$ and the topology of $X$ is  strictly weaker than the topology of $Y$. For other basic topological notions, one can see \cite{re}.

\section{\bf Metrizability of $C_h(X)$ and $C_{ph}(X)$}
In this section, we show that a number of properties of the spaces $C_h(X)$ and $C_{ph}(X)$ are equivalent to their metrizability. Some of these properties will be used in the sequel. We first define these properties.

A space $X$ is said to have countable pseudocharacter if each point in $X$ is a $G_{\delta}$-set. A subset $S$ of a space $X$ is said to have {\it{countable character}}\index{countable character} if there exists a sequence $\{W_n: n \in \mathbb{N}\}$ of open subsets in $X$ such that $S \subseteq W_n$ for all $n$ and if $W$ is any open set containing $S$, then $W_n \subseteq W$ for some $n$.

A space $X$ is said to be of \textit{countable type}\index{countable type}  (\textit{pointwise countable type})\index{pointwise countable type} if each compact set (point) is contained in a compact set having countable character. Clearly, every first countable space is of pointwise countable type.

A property weaker than being a space of pointwise countable type is that of being an r-space. A space $X$ is an \textit{r-space}\index{r-space} if each point of $X$ has a sequence $\{V_n:n\in \mathbb{N}\}$ of neighborhoods with the property that if $x_n\in V_n$ for each $n$, then the set $\{x_n:n\in \mathbb{N}\}$ is contained in a compact subset of $X$. Another property weaker than being an r-space is that of being a q-space. A space $X$ is a {\it{q-space}}\index{q-space} if for each point $x\in X$, there exists a sequence $\{U_n:n\in \mathbb{N}\}$ of neighborhoods of $x$ such that if $x_n\in U_n$ for each $n$, then $\{x_n:n\in \mathbb{N}\}$ has a cluster point.\\

Another property stronger than being a $q$-space is that of being an {\it{M-space}}\index{M-space}, which can be characterized as a space that can be mapped onto a metric space by a {\it{quasi-perfect map}}\index{quasi-perfect map} (a continuous closed map in which inverse images of points are countably compact). A space $X$ is called a {\it{p-space}}\index{p-space} if there exists a sequence $(\mathcal{U}_n)$ of families of open sets in a compactification of $X$ such that each $\mathcal{U}_n$ covers $X$ and $\bigcap_n\bigcup\{U \in \mathcal{U}_n : x \in U\} \subseteq X$ for any $x \in X$.\\
A metrizable space is of countable type and a space of pointwise countable type is an $r$-space. Also every metrizable space is a p-space and every p-space is a q-space.

In order to relate the metrizability of the spaces $C_h(X)$ and $C_{ph}(X)$ with the topological properties discussed above, we need the following
known lemma, the proof of which is omitted.
\begin{lemma}\label{pointwise_lemma}
Let $D$ be a dense subsets of a space $X$ and $A$ be a compact
subset of $D$. Then $A$ has countable character in $D$ if and only
if $A$ has countable character in $X$.
\end{lemma}

By Theorems 3.7 and 3.8 in \cite{AMK}, if $X^0$ is $G_{\delta}$-dense in $X$, then the spaces $C_{h}(X)$ and $C_{ph}(X)$ are topological groups and hence homogeneous. A space $X$ is called {\it{homogeneous}}\index{homogeneous space} if for every pair of points $x,\ y$ in $X$, there exists a homeomorphism of $X$ onto $X$ itself which carries $x$ to $y$.  So Lemma \ref{pointwise_lemma} can be used to prove the following result.

\begin{proposition}\label{dense_subspace_of_pointwise_countable_type}
If $X^0$ is $G_\delta$-dense in $X$, then $C_{\tau}(X)$, where $\tau = h, ph$, is of pointwise countable type if and only if
$C_{\tau}(X)$, where $\tau = h, ph$, has a dense subspace of pointwise countable type.
\end{proposition}

\begin{theorem}\label{metrizability of C_h(X)}
If $X^0$ is $G_\delta$-dense in $X$, then the following are equivalent.
\begin{enumerate}
 \item [$(a)$] $C_h(X)$ is metrizable.
 \item [$(b)$] $C_h(X)$ is first countable.
 \item [$(c)$] $C_h(X)$ has countable pseudocharacter.
 \item [$(d)$] $X$ has a countable $\pi$-base.
 \item [$(e)$] $X$ is a countable discrete space.
 \item [$(f)$] $C_p(X)$ is metrizable.
 \item [$(g)$] $C_h(X)$ is of countable type.
 \item [$(h)$] $C_h(X)$ is of pointwise countable type.
 \item [$(i)$] $C_h(X)$ has a dense subspace of pointwise countable type.
 \item [$(j)$] $C_h(X)$ is an r-space.
 \item [$(k)$] $C_h(X)$ is an M-space.
 \item [$(l)$] $C_h(X)$ is a p-space.
 \item [$(m)$] $C_h(X)$ is a q-space.
\end{enumerate}
\end{theorem}
\begin{proof}
\noindent The equivalences $(a)\Leftrightarrow (b) \Leftrightarrow (d)\Leftrightarrow (e)\Leftrightarrow (f)$ follow from Theorem 4.6 in \cite{AMK}.

\noindent $(b)\Rightarrow (c)$. Immediate.

\noindent $(c)\Rightarrow (d)$. This follows from Proposition 4.5 in \cite{AMK}.

\noindent $(a)\Rightarrow (g)\Rightarrow (h)\Rightarrow (j)\Rightarrow (m)$, $(a)\Rightarrow (k)\Rightarrow (m)$ and $(a)\Rightarrow (l)\Rightarrow (m)$ follow from the previous discussion. Also $(h)\Leftrightarrow (i)$ follows from Proposition \ref{dense_subspace_of_pointwise_countable_type}.

\noindent $(m)\Rightarrow (e)$. Since $X^0$ is dense in $X$, the constant zero function $0_X$ has a sequence $\{B_n:n\in \mathbb{N}\}$ of basic neighborhoods of the form $B_n=[\{x^n_1\},0]^-\cap\ldots\cap [\{x^n_{r_n}\},0]^-$ which satisfies the definition of $q$-space at $0_X$. Suppose that there exists $x_0\in X^0\setminus \cup\{A_n:n\in \mathbb{N}\}$, where $A_n=\{x^n_1,\ldots, x^n_{r_n}\}$. Then for each $n\in \mathbb{N}$, there is a $g_n\in C(X)$ such that $g_n(x_0)=n$ and $g_n(x)=0$ for each $x\in A_n$. Each $g_n\in B_n$. But $\{g_n:n\in \mathbb{N}\}$ does not cluster in $C_h(X)$, because for any $g\in C(X)$, $[\{x_0\},g(x_0)]^-$ is a neighborhood of $g$ in $C_h(X)$ which contains at most one member of the sequence $(g_n)_{n\in \mathbb{N}}$. Hence $X^0$ is countable. Since $X^0$ is $G_\delta$-dense, $X$ is a countable discrete space.
\end{proof}

\begin{theorem}\label{metrizability of C_ph(X)}
If $X^0$ is $G_\delta$-dense in $X$, then the following are equivalent.
\begin{enumerate}
  \item [$(a)$] $C_{ph}(X)$ is metrizable.
  \item [$(b)$] $C_{ph}(X)$ is first countable.
  \item [$(c)$] $C_{ph}(X)$ has countable pseudocharacter.
  \item [$(d)$] $X$ is separable.
  \item [$(e)$] $X$ is a countable discrete space.
  \item [$(f)$] $C_p(X)$ is metrizable.
  \item [$(g)$] $C_{ph}(X)$ is of countable type.
  \item [$(h)$] $C_{ph}(X)$ is of pointwise countable type.
  \item [$(i)$] $C_{ph}(X)$ has a dense subspace of pointwise countable type.
 \item [$(j)$] $C_{ph}(X)$ is an r-space.
 \item [$(k)$] $C_{ph}(X)$ is an M-space.
 \item [$(l)$] $C_{ph}(X)$ is a p-space.
 \item [$(m)$] $C_{ph}(X)$ is a q-space.
\end{enumerate}
\end{theorem}
\begin{proof}
\noindent The equivalences $(a)\Leftrightarrow (b) \Leftrightarrow (e)\Leftrightarrow (f)$ follow from Theorem 4.10 in \cite{AMK} and Theorem \ref{metrizability of C_h(X)}.

\noindent $(b)\Rightarrow (c)$. Immediate.

\noindent $(c)\Rightarrow (d)$. This follows from Proposition 4.5 in \cite{AMK}.

\noindent $(c)\Rightarrow (d)$. This follows from Theorem 4.8 in \cite{AMK1}.

\noindent $(d)\Rightarrow (e)$. If $X$ is separable, then $X^0$ is countable. Consequently, $X$ is a countable discrete space.

\noindent $(a)\Rightarrow (g)\Rightarrow (h)\Rightarrow (j)\Rightarrow (m)$, $(a)\Rightarrow (k)\Rightarrow (m)$ and $(a)\Rightarrow (l)\Rightarrow (m)$ follow from the previous discussion. Also $(h)\Leftrightarrow (i)$ follows from Proposition \ref{dense_subspace_of_pointwise_countable_type}.

\noindent $(m)\Rightarrow (d)$. Let $\{B_n:n\in \mathbb{N}\}$ be a sequence of basic neighborhoods of $0_X$ in $C_{ph}(X)$ which satisfies the definition of $q$-space at $0_X$. Without loss of generality, we can assume that each $B_n$ is of the form $[y^n_1,V^n_1]^+\cap\ldots\cap[y^n_{m_n},V^n_{m_n}]^+\cap[U^n_1,0]^-\cap\ldots\cap [U^n_{r_n},0]^-$. Choose $x^n_i\in U^n_i$ for $1\leq i\leq r_n$. Suppose that there exists $x_0\in X\setminus \cup\{A_n:n\in \mathbb{N}\}$, where $A_n=\{y^n_1,\ldots, y^n_{m_n}, x^n_1,\ldots, x^n_{r_n}\}$. Then for each $n\in \mathbb{N}$, there is a $g_n\in C(X)$ such that $g_n(x_0)=n$ and $g_n(x)=0$ for each $x\in A_n$. Each $g_n\in B_n$. But $\{g_n:n\in \mathbb{N}\}$ does not cluster in $C_{ph}(X)$, because for any $g\in C(X)$, $[x_0,(g(x_0)-\frac{1}{2},g(x_0)+\frac{1}{2})]^+$ is a neighborhood of $g$ in $C_{ph}(X)$ which contains at most one member of the sequence $(g_n)_{n\in \mathbb{N}}$. Hence $X$ is countable.
\end{proof}

\section{\bf Completeness properties of $C_h(X)$ and $C_{ph}(X)$}
In this section, we study various kinds of completeness of $C_h(X)$ and $C_{ph}(X)$. In particular, we look at the complete metrizability of the spaces $C_h(X)$ and $C_{ph}(X)$ in a wider setting, more precisely, in relation to several other completeness properties. We first find when these spaces are pseudocomplete and Baire.

A quasi-regular space $X$ is called pseudocomplete if it has a sequence of $\pi$-bases $\{ \mathcal{B}_n : n \in \mathbb{N} \}$ such that whenever $B_n \in \mathcal{B}_n$ for each $n$ and $\overline{B_{n+1}} \subseteq B_n$, then $\cap \{ B_n : n \in \mathbb{N} \} \neq \emptyset$ (\cite{JO}). In \cite{aarts_lutzer}, it has been shown that a pseudocomplete space is a Baire space.

The following results help us to find a necessary condition for $C_h(X)$ and $C_{ph}(X)$ to be Baire spaces.

\begin{proposition}\label{open notdense finite openset}
A nonempty basic open set $[U_1, r_1]^-\cap\ldots\cap[U_n, r_n]^-$ in $C_h(X)$ is dense in $C_h(X)$ if and
only if $U_i$ is an infinite subset of $X$ for each $i\in\{1,\ldots,n\}$.
\end{proposition}
\begin{proof}
Suppose that $[U_1, r_1]^-\cap\ldots\cap[U_n, r_n]^-$ is a nonempty basic open set in $C_h(X)$ such that $U_i$ is an infinite subset of $X$ for each $1\leq i\leq n$. Let $G=[V_1, t_1]^-\cap\ldots\cap[V_m, t_m]^-$ be any nonempty basic open set in $C_h(X)$. As $G\neq\emptyset$, there exists $f\in G$ such that $f(x_i)=t_i$, for some $x_i\in V_i$ and $1\leq i\leq m$. Since $U_i$ is an infinite subset of $X$, we can choose $y_i\in U_i$ for each $1\leq i\leq n$ such that $y_i\neq y_j$ for $i\neq j$ and $\{y_1, \ldots, y_n\}\cap \{x_1,\ldots,x_m\}=\emptyset$. Let $S=\{x_1, \ldots ,x_m, y_1,\ldots ,y_n\}$. Since $S$ is compact and $X$ is Tychonoff, there exists $h\in C(X)$ such that $h(x_i)=f(x_i)$ for each $i\in\{1,\ldots,m\}$ and $h(y_j)=r_j$ for each $j\in \{1,\ldots,n\}$. Thus $h\in [U_1, r_1]^-\cap\ldots\cap[U_n, r_n]^-\cap G$. Hence $[U_1, r_1]^-\cap\ldots\cap[U_n, r_n]^-$ is dense in $C_h(X)$.

Conversely, let $H=[U_1, r_1]^-\cap\ldots\cap[U_n, r_n]^-$ be a nonempty basic dense open set in $C_h(X)$. Suppose that for some $i\in \{1,\ldots,n\}$, $U_i$ is finite. Then $U_i\subseteq X^0$. Let $U_i=\{x_1,\ldots, x_m\}$. Take $r\in \mathbb{R}\setminus\{r_1,\ldots,r_n\}$ and $G=[x_1, r]^-\cap\ldots\cap[x_m, r]^-$. Then $G$ is a nonempty open subset of $C_h(X)$ such that $H\cap G=\emptyset$. So we arrive at a contradiction.
\end{proof}

\begin{corollary}\label{open notdense finite openset_coro} If $X$ is a space without isolated points, then every nonempty open set in $C_h(X)$ is dense in $C_h(X)$.\end{corollary}

\begin{proposition}\label{open notdense finite openset for Cph}
A nonempty open set of the form $[U_1, r_1]^-\cap\ldots\cap[U_n, r_n]^-$ in $C_{ph}(X)$ is dense in $C_{ph}(X)$ if and
only if if $U_i$ is an infinite subset of $X$ for each $i\in\{1,\ldots,n\}$.
\end{proposition}
\begin{proof}
 Suppose that $[U_1, r_1]^-\cap\ldots\cap[U_n, r_n]^-$ is a nonempty open set in $C_{ph}(X)$ such that $U_i$ is an infinite subset of $X$, for each $1\leq i\leq n$. Let $G=[z_1, H_1]^+\cap\ldots\cap[z_p, H_p]^+\cap[V_1, t_1]^-\cap\ldots\cap[V_m, t_m]^-$ be any nonempty basic open set in $C_{ph}(X)$. As $G\neq\emptyset$, there exists $f\in G$ such that $f(z_j)\in H_j$ for $1\leq j\leq p$ and $f(x_i)=t_i$, for some $x_i\in V_i$ and $1\leq i\leq m$. Since $U_i$ is infinite, we can choose $y_j\in U_j$ for $1\leq j\leq n$ such that for $i\neq j$, $y_i\neq y_j$ and $\{y_1, \ldots, y_n\}\cap \{x_1,\ldots,x_m,z_1,\ldots,z_p\}=\emptyset$. Let $S=\{x_1, \ldots ,x_m, y_1, \ldots ,y_n,z_1,\ldots,z_p\}$. Since $S$ is finite and $X$ is Tychonoff, there exists $h\in C(X)$ such that $h(x_i)=f(x_i)$ for each $i\in\{1,\ldots,m\}$, $h(z_k)=f(z_k)$ for each $k\in\{1,\ldots,p\}$ and $h(y_j)=r_j$ for each $j\in \{1,\ldots,n\}$. Thus $h\in [U_1, r_1]^-\cap\ldots\cap[U_n, r_n]^-\cap G$. Hence $[U_1, r_1]^-\cap\ldots\cap[U_n, r_n]^-$ is a dense open set in $C_{ph}(X)$.

Conversely, let $H=[U_1, r_1]^-\cap\ldots\cap[U_n, r_n]^-$ be a nonempty dense open set in $C_{ph}(X)$. This implies that $H$ is also a nonempty dense open set in $C_h(X)$. Then Proposition \ref{open notdense finite openset} implies that for each $1\leq i\leq n$, $U_i$ is an infinite subset of $X$.
\end{proof}

\begin{theorem}\label{firstcountablty nece Baire}
Suppose that $X$ has a non-isolated point $x_0$ such that $x_0$ is contained in a compact set having countable character.
Then $C_h(X)$ and $C_{ph}(X)$ are of first category, and hence are not Baire spaces.
\end{theorem}
\begin{proof}
Let $A$ be a compact set in $X$ such that $x_0\in A$ and $A$ has a countable character. Let $\mathcal{B}$ be a countable family of open sets in $X$ such that every member of $\mathcal{B}$ contains $A$ and if $U$ is an open set in $X$ containing $A$, then there exists a $B\in \mathcal{B}$ such that $B\subseteq U$. Since $x_0$ is a non-isolated point, for every $U\in \mathcal{B}$, $U$ is an infinite subset of $X$. So by Proposition \ref{open notdense finite openset} for every $U\in \mathcal{B}$ and $r\in \mathbb{R}$, $[U,r]^-$ is an open dense subset of $C_h(X)$. Thus $\mathcal{D}=\{([U,q]^-)^c:U\in\mathcal{B}, q\in\mathbb{Q}\}$ is a countable collection of closed and nowhere dense subsets of $C_h(X)$. In order to prove that $C_h(X)$ is of first category, we show that $C(X)=\cup \mathcal{D}$.

Suppose that there exists $f\in C(X)$ such that $f\notin \cup \mathcal{D}$. This implies that $f\in [U,q]^-$ for every $U\in \mathcal{B}$ and $q\in\mathbb{Q}$. Thus $\mathbb{Q}\subseteq f(U)$ for all $U\in\mathcal{B}$. Since $A$ is compact in $X$, there exists an open interval $(a,b)$ in $\mathbb{R}$ such that $f(A)\subseteq (a,b)$. Now $A$ has countable character, therefore there exists $U\in\mathcal{B}$ such that $A\subseteq U\subseteq f^{-1}(a,b)$. Consequently, $\mathbb{Q}\subseteq f(U)\subseteq (a,b)$, which is impossible. Hence $C_h(X)$ is not a Baire space.

By using Proposition \ref{open notdense finite openset for Cph} in place Proposition \ref{open notdense finite openset}, we can prove in a similar manner that $C_{ph}(X)$ is of first category.
\end{proof}

\begin{corollary}\label{corollaryfirstcountablty nece Baire}
If $X$ is a space of pointwise countable type, and either $C_h(X)$ or $C_{ph}(X)$ is a Baire space, then $X$ is discrete.
\end{corollary}

 In order to show that the converse of Corollary \ref{corollaryfirstcountablty nece Baire} is also true, we need the following results.

\begin{proposition}\label{comparison of topologies}
For any space $X$, the following are equivalent.
\begin{enumerate}
  \item [$(a)$] The space $C_p(X)< C_h(X)$.
  \item [$(b)$] The space $C_{ph}(X)= C_h(X)$.
  \item [$(c)$] $X$ is a discrete space.
\end{enumerate}
\end{proposition}
\begin{proof}
\noindent $(a)\Rightarrow (b)$. It is immediate.

\noindent $(b)\Rightarrow (c)$. Suppose that $X$ is not discrete. Then $X$ has some non-isolated point $x_0$. Let $0_X$ be the constant zero-function on $X$. Now $[x_0,(-1,1)]^+$ is an open neighborhood of $0_X$  in $C_{ph}(X)$. To show that $[x_0,(-1,1)]^+$ is not a neighborhood of $0_X$ in $C_h(X)$, let $$B=[U_1,0]^-\cap\ldots\cap[U_n,0]^-$$ be any basic neighborhood of $0_X$ in $C_h(X)$. Since $x_0$ is a non-isolated point, there exists $x_i\in U_i\setminus\{x_0\}$ for each $1\leq i\leq n$.  As $X$ is a Tychonoff space, there exists a continuous function $g:X\rightarrow [0,1]$ such that $g(x_0)=1$ and $g(x)=0$ for all $x\in \{x_1,\ldots,x_n\}$. Then $g\in B$ but $g\notin [x_0,(-1,1)]^+$. So $[x_0,(-1,1)]^+$ cannot be open in $C_h(X)$.

\noindent $(c)\Rightarrow (a)$. Suppose that $X$ is discrete. Take any subbasic open set $H$ in $C_p(X)$, where $H=[x,V]^+=\{f\in C(X):f(x)\in V\}$ for some $x\in X$ and some open set $V$ in $\mathbb{R}$. Since $\{x\}$ is open in $X$, for any $f\in [x,V]^+$, we have $f\in [x,f(x)]^-\subseteq [x,V]^+$. Hence $[x,V]^+$ is open in $C_h(X)$.
\end{proof}

\begin{lemma}\label{XdiscreteChXandRXHomeomorphism}
If $X$ is a discrete space, then $C_h(X)$ is homeomorphic to the product of $|X|$ many copies of the space $\mathbb{R}$ with discrete topology.
\end{lemma}
\begin{proof}
Suppose that $X=\{x_i:i\in I\}$ and let $\mathbb{R}_d$ denote the set $\mathbb{R}$ with discrete topology. Define $\psi:C_h(X)\rightarrow \mathbb{R}_{d}^{|X|}$ by $\psi(f)=(f(x_i))_{i\in I}$ for each $f\in C(X)$, and define $\phi:\mathbb{R}_{d}^{|X|}\rightarrow C_h(X)$ by $\phi((r_i)_{i\in I})=f$, where $f(x_i)=r_i$ for each $i\in I$ (note that $f\in C(X)$ since $X$ is a discrete space). It is easy to see that $\phi\circ \psi$ is the identity map on $C_h(X)$ and $\psi\circ \phi$ is the identity map on $\mathbb{R}_{d}^{|X|}$. So $\psi$ is a bijection. Since $X$ is a discrete space, it is easy to prove that the functions $\psi$ and $\phi$ are continuous.
\end{proof}

When $X$ is of pointwise countable type, the next result completely characterizes the situation when $C_{\tau}(X)$ ($\tau=h$, $ph$) is either pseudocomplete or a Baire space. In fact, in this case $C_{\tau}(X)$ ($\tau=h$, $ph$) is pseudocomplete or a Baire space if and only if $C_p(X)$ is so respectively.

\begin{theorem}\label{BaireandCechcompletepseudocomplete}
If $X$ is of pointwise countable type, then the following are equivalent.
\begin{enumerate}
\item[$(a)$] $C_h(X)$ is a Baire space.
\item[$(b)$] $C_{ph}(X)$ is a Baire space.
\item[$(c)$] $X$ is discrete.
\item[$(d)$] $C_h(X)$ is pseudocomplete.
\item[$(e)$] $C_{ph}(X)$ is pseudocomplete.
\item[$(f)$] $C_p(X)$ is pseudocomplete.
\item[$(g)$] $C_p(X)$ is a Baire space.
\end{enumerate}
\end{theorem}
\begin{proof}
\noindent $(a)\Rightarrow (c)$ and $(b)\Rightarrow (c)$. These follow from Corollary \ref{corollaryfirstcountablty nece Baire}.

\noindent $(c)\Rightarrow (d)$ and $(c)\Rightarrow (e)$. These follow from Proposition \ref{comparison of topologies}, Lemma \ref{XdiscreteChXandRXHomeomorphism} and the fact that every completely metrizable space is pseudocomplete and arbitrary product of pseudocomplete spaces is pseudocomplete (see \cite{JO}).

\noindent $(d)\Rightarrow (a)$, $(e)\Rightarrow (b)$ and $(f)\Rightarrow (g)$. These are true because every pseudocomplete space is a Baire space (see \cite{JO}).

\noindent $(c)\Rightarrow (f)$. If $X$ is discrete, then $C_p(X)=\mathbb{R}^X$. But an arbitrary product of pseudocomplete spaces is pseudocomplete.

\noindent $(g)\Rightarrow (c)$. This follows from Corollary I.3.6 in \cite{arhan}.
\end{proof}

Now we study the complete metrizability of the spaces $C_h(X)$ and $C_{ph}(X)$ in relation to several other completeness properties. We first recall the definitions of various kinds of completeness. For the rest of this section all spaces are assumed to be Tychonoff.

A space $X$ is called $\check{C}$ech-complete if $X$ is a $G_\delta$-set in $\beta X$, where $\beta X$ is the Stone-$\check{C}$ech compactification of $X$. A space $X$ is called locally $\check{C}$ech-complete if every point $x \in X$ has a $\check{C}$ech-complete neighborhood. Clearly, every $\check{C}$ech-complete space is locally $\check{C}$ech-complete.

In order to deal with sieve-completeness, partition-completeness, one needs to recall the definitions of these concepts from \cite{m4}. The central idea of all these concepts is that of a complete sequence of subsets of $X$.

Let $\mathcal{F}$ and $\mathcal{U}$ be two collections of subsets of $X$.  Then $\mathcal{F}$ is said to be controlled by $\mathcal{U}$ if for each $U\in \mathcal{U}$, there exists some $F\in \mathcal{F}$ such that $F\subseteq U$. A sequence $(U_n) $ of subsets of $X$ is said to be complete if every filter base $\mathcal{F}$ on $X$  which is controlled by $(U_n)$ clusters at some $x\in X$. A sequence $(\mathcal{U}_n)$ of collections of subsets of $X$ is called complete if $(U_n)$ is a complete sequence of subsets of $X$ whenever $U_n\in\mathcal{U}_n$ for all $n$. It has been shown in Theorem 2.8 of \cite{frolic} that the following statements are equivalent for a Tychonoff space $X$: $(a)$  $X$ is a $G_{\delta}$-subset of any Hausdorff space in which it is densely embedded; $(b)$ $X$ has a complete sequence of open covers; and $(c)$ $X$ is $\check{C}$ech-complete. From this result, it easily follows that a Tychonoff space $X$ is $\check{C}$ech-complete if and only if $X$ is a $G_{\delta}$-subset of any Tychonoff space in which it is densely embedded.\\

For the definitions of sieve, sieve-completeness and partition-completeness, see \cite{by},  \cite{m4},  and \cite{m5}. The term ``sieve-complete" is due to Michael \cite{michael}, but the sieve-complete spaces were studied earlier under different names: as $\lambda_b$-spaces by Wicke in \cite{wi}, as spaces satisfying condition {\large{$\mathscr{K}$}} by Wicke and Worrel Jr. in \cite{ww} and as monotonically $\check{C}ech$-complete spaces by Chaber, $\check{\mbox{C}}$oban and Nagami in \cite{CNN}. Every space with a complete sequence of open covers is sieve-complete; the converse is generally false, but it is true in paracompact spaces, see Remark 3.9 in \cite{CNN} and Theorem 3.2 in \cite{michael}. So a $\check{C}$ech-complete space is sieve-complete and a paracompact sieve-complete space is $\check{C}$ech-complete. Also every sieve-complete space is partition-complete.\\

\begin{theorem}\label{complete_metrizability_of_Ch(X)}
For any space $X$, the following are equivalent.
\begin{enumerate}
\item[$(a)$] $C_h(X)$ is completely metrizable.
\item[$(b)$] $C_h(X)$ is $\check{C}$ech-complete.
\item[$(c)$] $C_{h}(X)$ is locally $\check{C}$ech-complete.
\item[$(d)$] $C_{h}(X)$ is sieve-complete.
\item[$(e)$] $C_{h}(X)$ is an open continuous image of a paracompact $\check{C}$ech-complete space.
\item[$(f)$] $C_{h}(X)$ is an open continuous image of a $\check{C}$ech-complete space.
\item[$(g)$] $C_{h}(X)$ is partition-complete.
\item[$(h)$] $C_{p}(X)$ is completely metrizable.
\item[$(i)$] $C_{p}(X)$ is $\check{C}$ech-complete.
\item[$(j)$] $X$ is a countable discrete space.
\end{enumerate}
\end{theorem}
\begin{proof}
\noindent $(a)\Rightarrow(b)$. It follows from the fact that every completely metrizable space is $\check{C}$ech-complete.

\noindent $(b) \Rightarrow(c)$, $(b) \Rightarrow(d)$ and $(d) \Rightarrow (g)$ follows from the previous discussion. Also implications $(a) \Rightarrow(e)\Rightarrow(f)$ are immediate. Note $(c)\Rightarrow (f)$, see 3.12.19 $(d)$, page 237 in \cite{re}.

\noindent
$(f)\Rightarrow(a)$. A  $\check{C}$ech-complete space is of pointwise countable type and the property of being pointwise countable type is preserved by open continuous maps. Hence $C_{h}(X)$ is of pointwise countable type and consequently by Theorem \ref{metrizability of C_h(X)}, $C_{h}(X)$ is metrizable and hence $C_{h}(X)$ is paracompact. But a paracompact open image of a $\check{C}$ech-complete space is $\check{C}$ech-complete (see 5.5.8 (b), page 341 in \cite{re}). Hence $C_{h}(X)$ is $\check{C}$ech-complete. But a metrizable and $\check{C}$ech-complete space is completely metrizable.

\noindent $(g) \Rightarrow(a)$. If $C_{h}(X)$ is partition-complete, then by Propositions 4.4 and 4.7 in \cite{m4}, $C_{h}(X)$ contains a dense $\check{C}$ech-complete subspace. But a $\check{C}$ech-complete space is of pointwise countable type. Hence  by Theorems \ref{metrizability of C_h(X)}, $C_{h}(X)$ is metrizable. But by Theorem 1.5 of \cite{m3} and Proposition 2.1 in \cite{m4}, a metrizable space is completely metrizable if and only if it is partition-complete. Hence $C_{h}(X)$ is completely metrizable.

\noindent $(a)\Rightarrow (j)$. If $C_h(X)$ is completely metrizable, then Theorem 3.7 in \cite{AMK} implies that $X^0$ is $G_\delta$-dense in $X$. Hence by Theorem \ref{metrizability of C_h(X)}, $X$ is a countable discrete space.

\noindent $(j)\Rightarrow (a)$.  Since a countable product of completely metrizable spaces is completely metrizable,  Lemma \ref{XdiscreteChXandRXHomeomorphism} implies that $C_h(X)$ is completely metrizable.

\noindent $(h)\Leftrightarrow (i)\Leftrightarrow (j)$. This follows from Theorem 8.6 in \cite{LM}.
\end{proof}

The next theorem can be proved in a manner similar to Theorem \ref{complete_metrizability_of_Ch(X)}, except we need to use here Theorem \ref{metrizability of C_ph(X)} in place of Theorem \ref{metrizability of C_h(X)}.
\begin{theorem}
For any space $X$, the following are equivalent.
\begin{enumerate}
\item[$(a)$] $C_{ph}(X)$ is completely metrizable.
\item[$(b)$] $C_{ph}(X)$ is $\check{C}$ech-complete.
\item[$(c)$] $C_{ph}(X)$ is locally $\check{C}$ech-complete.
\item[$(d)$] $C_{ph}(X)$ is sieve-complete.
\item[$(e)$] $C_{ph}(X)$ is an open continuous image of a paracompact $\check{C}$ech-complete space.
\item[$(f)$] $C_{ph}(X)$ is an open continuous image of a $\check{C}$ech-complete space.
\item[$(g)$] $C_{ph}(X)$ is partition-complete.
\item[$(h)$] $C_{p}(X)$ is completely metrizable.
\item[$(i)$] $C_{p}(X)$ is $\check{C}$ech-complete.
\item[$(j)$] $X$ is a countable discrete space.
\end{enumerate}
\end{theorem}

\end{document}